\documentclass[10pt]{amsart}

\usepackage{color,graphicx,shortvrb}
\usepackage[latin 1]{inputenc}
\usepackage{mathtools} 
\usepackage[active]{srcltx} 
\usepackage[colorlinks, linkcolor=blue, citecolor=blue, urlcolor=blue,
breaklinks=true]{hyperref}

\usepackage{enumerate}

\usepackage[sort]{natbib}
\usepackage{url}

\newtheorem{theorem}{Theorem}[section]

\newtheorem{corollary}[theorem]{Corollary}

\newtheorem{lemma}[theorem]{Lemma}

\newtheorem{hypothesis}{Hypothesis}
\theoremstyle{definition}

\theoremstyle{remark}
\newtheorem{remark}[theorem]{Remark}

\def\RR{{\mathbb{R}}}

\def\11{\textbf{$1$}}

\def\C#1{ \mathcal{#1} }

\def\ir{\displaystyle{\int_{\RR^N}}}

\DeclareGraphicsExtensions{.jpg,.pdf,.png,.eps}

\DeclareMathOperator{\sgn}{sgn}

\newcommand{\ap}[1]{\left\langle#1\right\rangle}

\def\:{\colon}

\def\C{\mathbb{C}}

\def\R{\mathbb{R}}

\def\D{\mathcal{D}}

\def\d{\,\mathrm{d}}

\def\dx{\d x}
\def\dy{\d y}
\def \ddt{\frac{\mathrm{d}}{\mathrm{d}t}}
\def \ddt{\frac{\mathrm{d}}{\mathrm{d}t}}

\def\p{\partial}

\def\irN{\int_{\R^N}}

\def \ir3 {\int_{\R^3}}

\def\ir{\int_{\R}}

\usepackage{enumerate}

\title{Improved energy methods for nonlocal diffusion problems}
\author{José A. Cañizo}
\address{Departamento de Matemática Aplicada\\Universidad de
  Granada\\18071 Granada, Spain}
\email{canizo@ugr.es}

\author{Alexis Molino}
\address{Departamento de Análisis Matemático\\Universidad de
  Granada\\18071 Granada, Spain}
\email{amolino@ugr.es}

\date{April 2017}     

\begin{document}

\begin{abstract}
  We prove an energy inequality for nonlocal diffusion operators of
  the following type, and some of its generalisations:
  \begin{equation*}
    Lu (x) := \irN K(x,y) (u(y) - u(x)) \d y,
  \end{equation*}
  where $L$ acts on a real function $u$ defined on $\R^N$, and we
  assume that $K(x,y)$ is uniformly strictly positive in a
  neighbourhood of $x=y$. The inequality is a nonlocal analogue of the
  Nash inequality, and plays a similar role in the study of the
  asymptotic decay of solutions to the nonlocal diffusion equation
  $\partial_t u = L u$ as the Nash inequality does for the heat
  equation. The inequality allows us to give a precise decay rate of
  the $L^p$ norms of $u$ and its derivatives. As compared to existing
  decay results in the literature, our proof is perhaps simpler and
  gives new results in some cases.
\end{abstract}

\maketitle

\tableofcontents

\section{Introduction}
\label{sec:intro}

In this paper we develop \emph{energy methods} which are useful in the
study of some partial differential equations involving nonlocal
diffusion terms. We start by the basic example which is the following
integro-differential equation in convolution form:
\begin{equation}
  \label{eq:nonlocal}
  \p_t u(t,x) = \irN J(x-y) \big( u(t,y)  - u(t,x) \big) \d y,
  \qquad u(0,x) = u_0(x)
\end{equation}
where $t \geq 0$ is the time variable, $x \in \RR^N$ is the space
variable, $u = u(t,x) \in \R$ is the unknown, and $J$ is the
\emph{diffusion kernel}. Typically one assumes that $J$ is smooth,
nonnegative, radially symmetric, and with integral $1$; we also
mention a variety of models with different assumptions and variations
of \eqref{eq:nonlocal} in Section \ref{sec:applications}. Equation
\eqref{eq:nonlocal} and its relatives appear as a nonlocal version of
the usual diffusion equation $\p_t u = \Delta u$, and it is known that
\eqref{eq:nonlocal} approximates it when $J$ is close to a Delta
function (see Theorem \ref{thm:heat-approx-decay} and the remarks
before it).

We will apply energy methods to deal with nonlocal problems that not
necessarily involve a convolution. That is, problems of the form
\begin{equation}\label{gen-kernel}
  \p_t u(t,x) = \irN K(x,y) (u(t,y) - u(t,x)) \d y,
\end{equation}
where our main hypotheses on $K$ can be summarized as follows:
$K(x,y)$ is a nonnegative symmetric function with
$\sup_{y \in \RR^N}\irN K(x,y)dx \leq C_K$ and such that $K$ is
strictly positive in a neighborhood of the closet set
$\{x=y\}$. Furthermore, the symmetry of $K$ can be replaced by
integrability conditions (see Subsection \ref{no-convolution}). On the
other hand, observe that it makes sense to assume that $K(x,x)>0$
since in many models it means that the probability that individuals
remain for some time at the point where they are is positive.

As a particular application which motivates our arguments we consider
the nonlocal dispersal model proposed by \citet{CCEM} (see also
\cite{CEG-MM2011,CEG-MM2016,CEG-MM}):
\begin{equation}
  \label{dispersal-eq-intro}
  \p_t u(t,x) =
  \ir J\left(\frac{x-y}{g(y)}\right)
  \frac{u(t,y)}{g(y)}\d y-u(t,x),
  \qquad
  \text{in $\RR \times [0,\infty)$,}
\end{equation}
with a prescribed initial data $u(x,0)=u_0(x)$ defined on $\RR$. Here
$J$ is an even, positive, smooth function such that
$\int_{\RR} J(x)\d x = 1$ and $\hbox{supp } J = [-1,1]$, and $g$ is a
continuous positive function which accounts for the {\it {dispersal
    distance}} which depends on the departing point. In this model $u$
represents the spatial distribution of a certain species, and $g$
models the heterogeneity of the environment which can affect the
distribution of a species through space-dependent dispersal
strategies. For this model we are able to give an explicit rate of
decay of the $L^p$ norm of solutions, which is to our knowledge a new
result (see Theorem \ref{th-dispersal}).

\medskip The driving idea of our methods is that solutions to
\eqref{eq:nonlocal} behave in many ways like solutions to the heat
equation
\begin{equation}
  \label{eq:heat}
  \p_t u = \Delta u,
  \qquad u(0,x) = u_0,
\end{equation}
where as usual the Laplacian $\Delta$ acts only on the space variable
$x$ (see Theorem \ref{thm:heat-approx-decay} and the comments before
it). For more details we refer the reader to \cite{SunLiYang2011} for
the Cauchy problem, \cite{CER2009} for Dirichlet boundary conditions
(see also \cite{MR2016} in a more general framework) and
\cite{CERW2008} for Neumann boundary conditions. One important
property of \eqref{eq:heat} is the following time decay of solutions
(see for instance \cite{Giga}): there is a constant $C = C(N,p) > 0$
such that
\begin{equation}
  \label{eq:heat-decay}
  \|u\|_p^p \leq \big(
  \|u_0\|_p^{-p \gamma} + C \|u_0\|_1^{-p \gamma} t
  \big)^{-\frac{1}{\gamma}},
  \qquad
  \text{for all $t \geq 0$},
\end{equation}
which holds for any $1 < p < +\infty$ and any initial  condition
$u_0 \in L^1(\R^N) \cap L^p(\R^N)$ nontrivial, and where
\begin{equation*}
  \gamma := \frac{2}{N(p-1)}.
\end{equation*}
In fact, it still holds for $u_0 \in L^1(\R^N)$ and all $t > 0$ by
removing the term $\|u_0\|_p^{-p \gamma}$. Here and below,
$L^p(\R^N)$ denotes the usual Lebesgue space of $p$-integrable
functions on $\R^N$, with associated norm denoted by
$\|\cdot\|_p$. There are several ways of showing this decay and
regularization property for the heat equation. One of them is noticing
that the $L^p$ norms are \emph{Lyapunov functionals} for
\eqref{eq:heat}: if $u$ solves \eqref{eq:heat} with
$u_0 \in L^p(\R^N)$ then
\begin{equation}
  \label{eq:heat-Lp-evol}
  \ddt \|u\|_p^p = - \frac{4 (p-1)}{p} \irN \big| \nabla (u^{\frac{p}{2}}) \big|^2.
\end{equation}
One can then compare the right hand side term to $\|u\|_p$ by using
the Gagliardo-Nirenberg-Sobolev inequality (which in this particular
case is known as the Nash inequality \cite{nash1958continuity})
\begin{equation}
  \label{eq:GNS}
  \|v\|_{2} \leq C_{{N}} \ \|\nabla v\|_2^\theta\ \|v\|_1^{1-\theta},
\end{equation}
with
\begin{equation*}
  \theta := \frac{N}{N+2}.
\end{equation*}
This inequality is valid in any dimension $N$; in dimensions $N \geq 3$
it can easily be obtained as a consequence of the more familiar
Sobolev inequality $\|u\|_{2^*} \leq C \|\nabla u\|_2$, where $2^* :=
2N / (N - 2)$. By using \eqref{eq:GNS} with $v = u^{p/2}$ we obtain
for any $p \geq 2$ that
\begin{equation}
  \label{eq:GNS-heat}
  \irN \big| \nabla (u^{\frac{p}{2}}) \big|^2
  \geq
  C_{{N}}^{-\frac{2}{\theta}}\,
  \|u\|_p^{\frac{p}{\theta}} \,
  \|u\|_{\frac{p}{2}}^{-\frac{p(1-\theta)}{\theta}}
  \geq
  C_{{N}}^{-\frac{2}{\theta}}\,
  \|u\|_p^{p(1+\gamma)} \,
  \|u\|_{1}^{-p \gamma},
\end{equation}
where the last step is obtained through an interpolation of
$\|u\|_{p/2}$ between $\|u\|_p$ and $\|u\|_1$. Due to mass
conservation for the heat equation we have $\|u\|_1 \leq \|u_0\|_1$
for all times $t \geq 0$ (this inequality is of course an equality for
nonnegative, finite-mass solutions). Hence using \eqref{eq:GNS-heat}
in \eqref{eq:heat-Lp-evol} one has
\begin{equation*}
  \ddt \|u\|_p^p \leq - C \, \|u\|_p^{p(1+\gamma)} \, \|u_0\|_1^{-p \gamma},
\end{equation*}
for some constant $C = C(N,p)$. This is a differential inequality for
$\|u\|_p$ that readily gives the decay \eqref{eq:heat-decay}.

In the context of diffusion equations, the strategy of using the $L^p$
norm of $u$ and its derivative as a means for studying properties of
solutions is known as the \emph{energy method}. It is a close relative
of a common and quite successful strategy in kinetic equations and
dissipative PDE sometimes known as the \emph{entropy method}
\citep{MR1760620, Gross, citeulike:2859252, CJMTU, citeulike:8176023,
  citeulike:8175960, Villani02, DV2004}, where one compares the time
derivative of a Lyapunov functional with the Lyapunov functional
itself via a functional inequality in order to obtain a certain decay
rate for solutions. These energy methods have the advantage of being
quite robust, often being applicable to equations that are not
explicitly solvable by Fourier transform methods, and to nonlinear
problems. The question that motivates this paper is whether these
ideas can be adapted to equation \eqref{eq:nonlocal} in order to show
a decay property similar to \eqref{eq:heat-decay}. One important
observation is that the same statement cannot be true for solutions of
\eqref{eq:nonlocal}, since there is no instantaneous $L^1$ to $L^p$
regularization. In fact, the $L^p$ norms are still a Lyapunov
functional for \eqref{eq:nonlocal} (as is well known, any convex
function gives a Lyapunov functional for \eqref{eq:nonlocal}): if $u$
is an $L^p$ solution to \eqref{eq:nonlocal} then
\begin{equation}
  \label{eq:nonlocal-Lp-evol}
  \ddt \|u\|_p^p = - \mathcal{D}_p^J(u).
\end{equation}
Here, the \emph{$L^p$ dissipation} $\D_p^J(u)$ is defined for any measurable
$u\: \R^N \to \R$ as
\begin{equation}
  \label{eq:Dp2}
  \D_p^J(u) := \frac{p}{2} \irN \irN J(x-y) \big(u(x) - u(y)\big)
  \big(\phi_{p-1}(u(x)) - \phi_{p-1}(u(y)) \big) \d x \d y,
\end{equation}
where for $q > 0$ we denote by $\phi_q$ the antisymmetric extension of
the usual $q$-th power, that is,
\begin{equation*}
  \phi_q(s) := |s|^{q} \sgn(s),
  \qquad s \in \R.
\end{equation*}
Of course,
since $\phi_{p-1}$ is nondecreasing, the integrand in \eqref{eq:Dp2}
is also nonnegative and always makes sense as a number in
$[0,+\infty]$. We point out that for nonnegative $u$ the expression
becomes a bit simpler,
\begin{equation*}
  \D_p^J(u) := \frac{p}{2} \irN \irN J(x-y) \big(u(x) - u(y)\big)
  \big(u(x)^{p-1} - u(y)^{p-1}\big) \d x \d y.
\end{equation*}
Precisely this strategy was discussed in \cite{citeulike:5333882},
where it was remarked that no inequality of the following form can
hold, for any $q > 2$ and a smooth, nonnegative, compactly supported
function $J$:
\begin{equation*}
  \D_2^J(u) \geq C \|u\|_q^2.
\end{equation*}
Hence the natural analogue of the usual Sobolev inequality does not
hold in the nonlocal case. Similarly, the direct analogue of
\eqref{eq:GNS-heat} (with $\D_p^J(u)$ on the left hand side) cannot
hold, since it would imply an $L^1 - L^p$ regularization effect on
\eqref{eq:nonlocal} which is known to fail. In view of this failure, a
different strategy was followed there, leading to different
inequalities and applications to several linear and nonlinear
equations involving nonlocal diffusions. Similar ideas were developed
in \cite{BP13} in order to establish decay estimates for fractional diffusions, with modified inequalities used in place of the
usual Nash inequality. After the statement of our results we compare
them in more detail to those in \cite{citeulike:5333882,BP13} and
other previous works.

\medskip \noindent
\textbf{Main results.} Our purpose is to show a simple inequality that
plays the role of \eqref{eq:GNS-heat} and provides a means to show
precise decay properties of \eqref{eq:nonlocal} and \eqref{gen-kernel}:
\begin{hypothesis}
  \label{hyp:J}
  $J \: \R^N \to [0,+\infty)$ is a measurable function such that
  for some $r, R > 0$ we have
  \begin{equation}
    \label{eq:J-lower-bound}
    J(z) \geq r, \quad \text{for all $|z| < R$}\,.
  \end{equation}
\end{hypothesis}
In particular, this is obviously satisfied if $J$ is continuous in a 
neighborhood of $0$ with $J(0) > 0$.
\begin{theorem}[$L^p$ energy inequality]
  \label{teo-principal}
  Let $J:\R^N \to \R$ be a function satisfying Hypothesis \ref{hyp:J}.
  For every $N\geq 1$ and $p\geq 2$, there exists a positive constant
  $C=C(N,p) > 0$ such that
  \begin{equation}\label{ineq-principal}
    \D_p^J(u) \geq
    C r \min \left\{
      R^{N+2} \|u\|_1^{-p \gamma}\, \|u\|_p^{p(1+\gamma)},\,
      R^N \|u\|_p^p
    \right\},
  \end{equation}
  for all $u \in L^1(\R^N)\cap L^p(\R^N)$, where
  $\gamma := \frac{2}{N(p-1)}$.
\end{theorem}

This inequality serves as a useful analogue of \eqref{eq:GNS-heat} in
the nonlocal case, as we will see shortly. If one does not care about
the precise dependence of the constant $C$ on $J$ then this can be
more simply stated as: there exists a constant $C = C(N,p,J)$
depending only on $N$, $p$ and $J$ such that
\begin{equation}\label{ineq-principal-2}
  \D_p^J(u) \geq
  C \min \left\{
    \|u\|_1^{-p \gamma}\, \|u\|_p^{p(1+\gamma)},\,
    \|u\|_p^p
  \right\}.
\end{equation}
The constants in the above inequalities can be estimated explicitly by
following the proof. To our knowledge, inequality
\eqref{ineq-principal} is new. Similar modified Nash inequalities are
considered in \cite{Carlen1987,citeulike:5333882}, and especially in
\cite{BP13}[Corollary 4.7]. In the latter, $(p,q)$-inequalities involving the $p$ and
$q$ norms of $u$ are given for $p > q > 1$; ours is the limiting
case with $q=1$, not included there. We notice the $L^1$ case is
fundamental for the generalisations we describe later, since mass is a
natural conserved quantity in many models.

The inequality in Theorem \ref{teo-principal} immediately allows us to
deduce bounds on the asymptotic behaviour of several nonlocal
diffusion equations (see Section \ref{sec:applications}). Let us give
the argument for equation \eqref{eq:nonlocal}, which is the simplest
possible model: using \eqref{eq:nonlocal-Lp-evol} we have
\begin{equation*}
  \ddt \|u\|_p^p = - \D_p^J(u)
  \leq
  -
  C r \min \left\{
    R^{N+2} \|u\|_1^{-p \gamma}\, \|u\|_p^{p(1+\gamma)},\,
    R^N \|u\|_p^p
  \right\}.
\end{equation*}
Taking into account that $\|u\|_1$ is nonincreasing in time (it is
conserved for nonnegative solutions) one has
\begin{equation*}
  \ddt \|u\|_p^p
  \leq
  -
  C r \min \left\{
    R^{N+2} \|u_0\|_1^{-p \gamma}\, \|u\|_p^{p(1+\gamma)},\,
    R^N \|u\|_p^p
  \right\}.
\end{equation*}
This is a differential inequality for $\|u\|_p$, which can be solved
(see Lemma \ref{lem:ode}) to yield the following result:
\begin{theorem}
  \label{thm:decay-main}
  Take a function $J\in L^1(\RR^N)$ satisfying Hypothesis \ref{hyp:J} and
  $p \in [2, +\infty)$. Consider the solution $u$ to equation
  \eqref{eq:nonlocal} with initial data
  $u_0 \in L^1(\R^N) \cap L^p(\R^N)$. There exists a constant
  $C = C(N, p)$ (the same as in Theorem \ref{teo-principal}) such that
  \begin{equation}
    \label{eq:decay-main}
    \|u\|_p^p \leq
    \begin{cases}
      \|u_0\|_p^p
      &\qquad \text{for $0 \leq t \leq t_0$},
      \\
      \big(
      \|u_0\|_p^{-p\gamma}
      + C \gamma r R^{N+2} \|u_0\|_1^{-p \gamma}
      (t-t_0)
      \big)^{-\frac{1}{\gamma}}
      &\qquad \text{for $t \geq t_0$},
    \end{cases}
  \end{equation}
  where $\gamma := \frac{2}{N(p-1)}$ and
  \begin{equation*}
    t_0 = \max\left\{
      0,\
      \frac{1}{C r R^N}
      \log \big( R^{\frac{2}{\gamma}} \|u_0\|_1^{-p} \|u_0\|_p^p ) \big)
    \right\}.
  \end{equation*}
\end{theorem}
Again, if we are not interested in the precise dependence of the bound on
$J$, $\|u_0\|_1$ and $\|u_0\|_p$ then the following statement is
simpler: there exists a constant
$C = C(r, R, N, p, \|u_0\|_1, \|u_0\|_p)$ such that
\begin{equation}
  \label{eq:decay-main-simple}
  \|u\|_p^p \leq C (1+t)^{-\frac{N(p-1)}{2}} \qquad \text{for all $t \geq 0$}.
\end{equation}
This is a direct consequence of the bound \eqref{eq:decay-main}; see
Remark \ref{rem:X}. In this sense, Theorem \ref{teo-principal} is a
nonlocal analogue of the Gagliardo-Nirenberg-Sobolev (or Nash)
inequality: it allows us to give a decay rate of the nonlocal
diffusion equation \eqref{eq:nonlocal}, and in fact this decay rate
approaches that of the heat equation as \eqref{eq:nonlocal} approaches
it (see Theorem \ref{thm:heat-approx-decay}). Furthermore, due to the
interpolation formula and using inequality \eqref{eq:decay-main-simple} for $p=2$, we obtain that for $q\in[1,2]$
\begin{align*}
\|u\|_q^q & \leq \|u\|_1^{2-q}\|u\|_2^{2(q-1)}
\\
& \leq \|u\|_1^{2-q}C^{q-1}(1+t)^{-\frac{N(q-1)}{2}}
\\
& \leq \tilde C (1+t)^{-\frac{N(q-1)}{2}},
\end{align*}
which means that \eqref{eq:decay-main-simple} also holds for
$1\leq p < 2$ and some positive constant
$ C = C(J, N, p, \|u_0\|_1, \|u_0\|_2)$.

We also give inequalities related to higher derivatives of $u$ in
Section \ref{sec:ineq-deriv}, and deduce from them corresponding decay
properties of derivatives of $u$, still at the same asymptotic rate as
those for the heat equation. For $k \geq 0$ we define the differential
operator $D^k$ acting on a function $u$ as
\begin{equation*}
  D^k u:=-(-\Delta)^{k/2} u.
\end{equation*}
In order to ensure that this expression makes sense we will always
assume that $u \in H^k(\R^N)$ (i.e., the classical Sobolev space
$W^{k,2}(\RR^N)$) when applying $D^k$. The following result gives an
estimate of $\D_2^J(D^ku)$; note that the case $k=0$ is just the $p=2$
case of Theorem \ref{teo-principal}:

\begin{theorem}
  \label{k derivative estimate}
  Let $N \geq 1$ be an integer and $J:\R^N \to \R$ be a function
  satisfying Hypothesis \ref{hyp:J}. There exists a positive constant
  $C=C(N)$ such that
  \begin{equation}\label{ineq-k-derivates}
    \D_2^J(D^k u) \geq
    C r \min \left\{ R^{k + N + 2} \|u\|_1^{-\frac{4}{N+2k}}\, \|D^k u\|_2^{2+\frac{4}{N+2k}},
      \,\, R^{k + N} \|D^k u\|_2^2
    \right\}
  \end{equation}
  for all $u \in H^k(\R^N) \cap L^1(\R^N)$.
\end{theorem}

As a consequence one can obtain a decay of higher derivatives of
solutions to \eqref{eq:nonlocal}. Notice that the case $k=0$ of the
following result is just Theorem \ref{thm:decay-main} with $p=2$:

\begin{theorem}
  \label{thm:decay-deriv}
  Take a function $J$ satisfying Hypothesis \ref{hyp:J} and a real
  $k \geq 0$. Consider the solution $u$ to equation
  \eqref{eq:nonlocal} with initial data
  $u_0 \in L^1(\R^N) \cap H^k(\R^N)$. There exists a constant
  $C = C(N, k)$ (the same as in Theorem \ref{k derivative estimate})
  such that
  \begin{equation*}
    \| D^k u \|_2^2
    \leq
    \begin{cases}
      \| D^k u_0 \|_2^2
      &\qquad \text{for $0 \leq t \leq t_0$},
      \\
      \big(
      \|D^k u_0 \|_2^{-2 \gamma}
      + C r \gamma R^{k+N+2} \|u_0\|_1^{-2 \gamma}
      (t-t_0)
      \big)^{-\frac{1}{\gamma}}
      &\qquad \text{for $t \geq t_0$},
    \end{cases}
  \end{equation*}
  where $\gamma := \frac{2}{N + 2k}$ and
  \begin{equation*}
    t_0 = \max \left\{
      0, \
      \frac{1}{C r R^{k+N}}
      \log \left(
        R^{\frac{2}{\gamma}} \|u_0\|_1^{-2}
        \|D^k u_0\|_2^2
      \right)
    \right\}.
  \end{equation*}
\end{theorem}

\medskip The decay in Theorem \ref{thm:decay-main} can be interpreted
as follows: for large times, the asymptotic decay of the $L^p$ norm
of solutions to the nonlocal diffusion equation \eqref{eq:nonlocal} is
the same as that of the heat equation. However, there can be an
initial time during which a different decay takes place. The threshold
between the two is related to the value of the $L^p$ norm of $u$: if
it is large then heuristically (since we are assuming $u_0$ is
integrable) the main contribution to the $L^p$ norm comes from local
concentrations of $u$. Since the smoothing effect of
\eqref{eq:nonlocal} is much weaker than that of the heat equation, the
rates of decay of the two differ. On the other hand, when $\|u\|_p$ is
small, the concentrations of $u$ do not play a major role and the
decay of both equations becomes comparable. The inequality
\eqref{ineq-principal} and the corresponding decay
\eqref{eq:decay-main} are quite precise on the dependence on $J$ and
the initial data, giving a direct estimate of the time when the
``heat-like'' diffusion kicks in: the time $t_0$ depends
logarithmically on the ratio between $\|u_0\|_p$ and $\|u_0\|_1$.

\medskip
Theorem \ref{thm:decay-main} as stated is not new; the simplified
statement \eqref{eq:decay-main-simple} can be proved for example by
Fourier transform methods \citep{AMRT2010}, and the decay
\eqref{eq:decay-main} can probably be obtained as well. The important
advantage of using Theorem \ref{teo-principal} to prove Theorem
\ref{thm:decay-main} is that the method is quite robust under
modifications in the linear operator. In Subsection
\ref{no-convolution} we prove a result similar to Theorem
\ref{thm:decay-main} which gives decay properties for general nonlocal
diffusion equations with a more general kernel $K(x,y)$ instead of
$J(x-y)$: consider the equation
\begin{equation}
  \label{general-kernel}
  \p_t u(t,x) = \irN K(x,y) u(t,y) \d y - \sigma(x) u(t,x),
\end{equation}
where $K \: \RR^N \times \RR^N \to [0,\infty)$ is a general kernel
(not necessarily symmetric) and $\sigma \: \R^N \to [0,+\infty)$ is a
function.  Let us keep our discussion at a formal level for the moment
and not worry about the problem of existence of solutions to
\eqref{general-kernel} or the precise regularity of $K$ and
$\sigma$. Equation \eqref{general-kernel} is a general form of the
\emph{scattering equation} (see for example \cite{MMP2004-general}),
and contains many others as a particular case. The nonlocal diffusion
\eqref{eq:nonlocal} is recovered if $K(x,y) = J(x-y)$ and
$\sigma(x) = \int_{\R^N} J$ for all $x,y$. In the case that
$\sigma(x) = \int_{\R^N} K(x,y) \d y$ the equation can be written as
\begin{equation}
  \label{general-kernel-diffusion}
  \p_t u(t,x) = \irN K(x,y) \big( u(t,y) - u(t,x) \big) \d y,
\end{equation}
which is a type of nonlocal diffusion equation, where the nonlocality
is not given by a convolution. Similarly, if we assume
\begin{equation}
  \label{eq:mass-conservation}
  \sigma(x) = \int_{\R^N} K(y,x) \d y,
\end{equation}
then equation \eqref{general-kernel} is formally the Kolmogorov
forward equation for a Markov jump process with jump rates given by
$K$, where $u$ represents the probability density of the process
\citep[Chapter 4.2]{ethier-kurtz86}. Notice that
\eqref{eq:mass-conservation} is just the statement that the total mass
$\int_{\R^N} u(t,x) \d x$ is formally conserved in time (as should
happen for a probabilistic evolution). In that sense, equation
\eqref{general-kernel} contains many evolution equations linked to
Markov processes, and has multiple applications. (We give an example
linked to a population dispersal in Section \ref{sec:dispersal}.)
Equation \eqref{general-kernel} has some properties in common with
diffusion processes, but it is important to notice that
\eqref{general-kernel} may have finite-mass equilibria (unlike the
usual heat equation, whose only finite-mass equilibrium is $0$).

Let us state a precise result which is relevant for nonlocal
diffusions. For all of them we will assume:
\begin{hypothesis}\label{eq:hyp1-general}
  \text{There exist $r, R > 0$ such that $K(x,y) \geq r$ whenever
    $|x-y| < R$}.
\end{hypothesis}
This is the analogue of Hypothesis \ref{hyp:J} in this setting. In
order to ensure that $L^p$ solutions of \eqref{general-kernel} exist
we will also assume that $K$ is measurable and that for some $C_K > 0$
\begin{equation}
  \label{eq:K-existence}
  \int_{\R^N} K(x,y) \d y \leq C_K,
  \quad
  \int_{\R^N} K(y,x) \d y \leq C_K,
\qquad \text{for all $x \in \R^N$}.
\end{equation}
This ensures that the linear operator on the right hand side of
\eqref{general-kernel} is bounded in $L^1(\R^N)$ and $L^\infty(\R^N)$
(and hence, by interpolation, in any $L^p(\R^N)$ with
$1 \leq p \leq \infty$).

\begin{theorem}
  \label{thm:decay-general}
  Take $p \in [2,+\infty)$. Assume that
  $K \: \R^N \times \R^N \to [0,+\infty)$ satisfies
  Hypothesis \ref{eq:hyp1-general} and \eqref{eq:K-existence}. Consider
  equation \eqref{general-kernel} with $\sigma$ given by
  \eqref{eq:mass-conservation}, and assume that there exists an
  equilibrium $u_\infty$ of \eqref{general-kernel} satisfying
  \begin{equation}
    \label{eq:equilibrium-bounded}
    \frac{1}{m} \leq u_\infty(x) \leq m,
    \qquad \text{for all $x \in \R^N$,}
  \end{equation}
  for some $m > 0$. Let $u$ be any solution to equation
  \eqref{general-kernel} with initial data
  $u_0 \in L^1(\R^N) \cap L^p(\R^N)$. There exists a constant $C$
  depending only on $r$, $R$, $N$, $m$, $p$, $\|u_0\|_1$ and
  $\|u_0\|_p$ such that
  \begin{equation*}
    \|u\|_p^p \leq C (1+t)^{-\frac{N(p-1)}{2}},
     \qquad \text{for all $t \geq 0$}.
  \end{equation*}
\end{theorem}

In Section \ref{sec:dispersal} we give an application of these results
to a dispersal equation proposed in \citet{CCEM}, obtaining an explicit
rate of convergence to equilibrium.

\begin{remark}
  Condition \eqref{eq:K-existence} is just included in order to ensure
  that there are well-defined solutions to \eqref{general-kernel}, but
  it does not play a role in the decay estimates. It can be removed if
  it can be justified by other means that solutions to
  \eqref{general-kernel} exist and rigorously satisfy the entropy
  property \eqref{eq:nonlocal-Lp-evol}.
\end{remark}

\begin{remark}
  In Theorem \ref{thm:decay-general} one can also give a more precise
  estimation of the decay and the constants involved, as we did in
  Theorem \ref{thm:decay-main}. We have preferred in this case to
  leave the statement in this form for simplicity, but the reader can
  state the analogue of Theorem \ref{thm:decay-main} without
  difficulty.
\end{remark}

We refer to Section \ref{no-convolution} for details on this and a
proof of Theorem \ref{thm:decay-general}.

\medskip
\noindent\textbf{Heat equation scaling.}
It is worth mentioning that Theorems \ref{teo-principal} and
\ref{thm:decay-main} pass to the limit well when the nonlocal equation
\eqref{eq:nonlocal} approximates the heat equation. Let $J$ be a
smooth and radially symmetric convolution kernel
with $J(0) > 0$, and denote by $J_\epsilon$ the rescaling
\[
J_\varepsilon(z)
:=\frac{C(J)}{\varepsilon^{2+N}} J\left(\frac{z}{\varepsilon}\right),
\qquad
\hbox{ with } C(J)^{-1} = \frac{1}{2}\int_{\RR^N}J(z)z_N^2 \d z.
\]
It is well-known that, $u^\varepsilon$, the solution to the equation
\begin{equation}
  \label{rescaled}
  \partial_t u^\varepsilon(t,x)
  =\int_{\RR^N} J_\varepsilon(x-y)
  ((u^\varepsilon(t,y)-u^\varepsilon(t,x)) \d y,
  \qquad x\in \RR^N, \, t>0,
\end{equation}
with initial data $u_0\in \mathcal{C}(\RR^N)$ converges to the
solution of the heat equation $\partial_t v = \Delta v$ with the same
initial data (see for instance \cite{AMRT2010,Toscani}). Since $J$
satisfies Hypothesis \ref{hyp:J} for some $r, R > 0$ one has
$J_\varepsilon(z) \geq \frac{rC(J)}{\varepsilon^{2+N}}$, for all
$|z|<R\varepsilon$. Replacing this in expression \eqref{eq:decay-main}
the $\varepsilon$ is cancelled and we obtain the following result:

\begin{theorem}
  \label{thm:heat-approx-decay}
  Assume $J$ satisfies Hypothesis \ref{hyp:J}. Let $u^\varepsilon$ be
  a solution of \eqref{rescaled} with initial data
  $u_0 \in L^1(\RR^N)\cap L^p(\RR^N)$ with $p\in [2,\infty)$. Then it
  holds
  \begin{equation*}
    \|u^\varepsilon(t,\cdot)\|_p^p \leq
    \big(
    \|u_0\|_p^{-p\gamma}
    + C_1 \|u_0\|_1^{-p \gamma}
    (t-t_0)
    \big)^{-\frac{1}{\gamma}}
    \qquad
    \text{for $t \geq t_0$},
  \end{equation*}
  where $C_1=C(N,p) \gamma r R^{N+2} C(J)$ does not depend on
  $\varepsilon$ and
  \[
  t_0 = \max\left\{
    0,\
    \frac{\varepsilon^2}{C r R^N C(J)}
    \log \big(\varepsilon^{\frac{2}{\gamma}} R^{\frac{2}{\gamma}}
    \|u_0\|_1^{-p} \|u_0\|_p^p ) \big)
  \right\}.
  \]
  In particular, $t_0=0$ for all
  $\varepsilon <\varepsilon_0 = \|u_0\|_1^{\frac{\gamma p}{2}} / \big(R
  \|u_0\|_p^{\frac{\gamma p}{2}}\big)$.
\end{theorem}
The interest of the above theorem is that the decay is preserved in
the scaling that leads to the heat equation. In addition, for small
$\varepsilon$ the expression of the decay is exactly of the same form as
that of the heat equation, given in \eqref{eq:heat-decay}.

\medskip
\noindent \textbf{Comparison to results in the literature.} Several
precise results exist already regarding the decay properties of
equation \eqref{eq:nonlocal}. Let us give a brief review and compare
them to our own. Nonlocal diffusions including \eqref{eq:nonlocal}
have been studied in \cite{CCR06}, and we refer the reader to the
recent book \cite{AMRT2010} for background and an extensive review of
the state of the art for equations involving similar nonlocal terms. A
similar approximation to the heat equation, with a particular kernel
$J$, was studied in \cite{Toscani}, and some nonlocal approximations
to Fokker-Planck equations have been recently considered in
\cite{MT2016} and very recently in \cite{T2017}.

The observation that solutions to
\eqref{eq:nonlocal} decay asymptotically like the heat equation has
been present since the first works on the matter, with several
analogues of \eqref{eq:heat-decay}. The first ones were based on the
Fourier transform of \eqref{eq:nonlocal}, which is explicitly solvable
\cite{CCR06,IR2007,citeulike:3617261}. Energy methods were considered
in \cite{citeulike:5333882}; results were given on the decay of
several models including the linear nonlocal diffusion equation
\eqref{eq:nonlocal} and a nonlocal version of the $p$-Laplacian
evolution equation. The method in \cite{citeulike:5333882} is
different from ours, and is based on a splitting of the function $u$
into a ``smooth'' part and a ``rough'' part. The ideas are somehow
reminiscent of ours, since they borrow techniques from Fourier
splitting by \cite{Schonbek} and there is a parallel with our
splitting of the function $u$ in Fourier space. The results from
\cite{citeulike:5333882} are in dimensions $N \geq 3$ and $K$
symmetric; on the other hand, they are well-adapted to nonlinear
problems like the nonlocal $p$-Laplacian equation. Our inequality
seems to be a simpler argument which works in any dimension, is
well-adapted to the linear nonlocal diffusion operator, but does not
easily carry over to nonlinear nonlocal operators. It also gives a
simple way to track the dependence of the decay on the parameters of
the problem, especially the diffusion kernel $J$.

Inequalities of the type \eqref{ineq-principal} were already noticed
in \cite{BP13}, and used in order to obtain decay and regularisation
properties for nonlinear diffusions of the type \eqref{eq:nonlocal}
where the function $J$ typically behaves as $|x|^{-N-\alpha}$ as
$x \to +\infty$, for some $0 < \alpha \leq 2$. Their proof goes
along the lines of \cite{citeulike:5333882}. Inequality
\eqref{ineq-principal} is a limit case of their results, but is not
included there for similar reasons as in \cite{citeulike:5333882}.

As compared to previous results, we summarise our contributions as
follows:
\begin{enumerate}
  
\item Inequality \eqref{ineq-principal} seems to be new. Similar ideas
  were used in \cite{citeulike:5333882,BP13}, but
  \eqref{ineq-principal} is a limiting case not included in these
  works.
  
\item Our proof of the inequality \eqref{ineq-principal} is
  straightforward, works in any dimension, and in our opinion
  simplifies previous arguments for related inequalities. It also
  leads to a precise estimate of the constants in the inequality,
  which have in particular the correct scaling when approximating the
  heat equation (see Theorem \ref{thm:heat-approx-decay}).
  
\item A similar method of proof yields inequalities and decay results
  involving higher derivatives of the function $u$; see Section
  \ref{sec:ineq-deriv}.
  
\item The entropy method used allows for an extension to linear
  mass-conserving equations with general kernels $K(x,y)$ (not
  necessarily symmetric) instead of $J(x-y)$; see Subsection
  \ref{no-convolution}.
  
\end{enumerate}

The paper is organised as follows: in Section \ref{sec:inequalities}
we give the proof of the inequality in Theorem \ref{teo-principal},
and in Section \ref{sec:ineq-deriv} we prove similar inequalities
involving derivatives. Finally, in Section \ref{sec:applications} we
show how these inequalities yield decay properties for several
equations involving general kernels $K(x,y)$, in particular proving
Theorem \ref{thm:decay-main} in Subsection \ref{sec:linear}.

\section{Energy inequalities for nonlocal diffusion operators}
\label{sec:inequalities}

We are interested in finding useful lower bounds of $\D_p^J(u)$ in
terms of $L^p$ norms of $u$. Since
$(|a|-|b|)(|a|^s - |b|^s) \leq (a-b) (\phi_s(a) - \phi_s(b))$ for any
$a, b \in \R$ and $s > 1$ (where $\phi_s(a) := |a|^s \sgn(a)$), it is
easily seen that
\begin{equation*}
  \D_p^J(u) \geq \D_p^J(|u|)
\end{equation*}
for any measurable $u \: \R^N \to \R$. This allows us to work only
with nonnegative functions $u$.

This section is devoted to the proof of Theorem
\ref{teo-principal}. We first show the case $p=2$, and then deduce
from it the general inequality for $p \geq 2$. The proof of the $p=2$
case is a modification of a the original proof of the Nash inequality
\eqref{eq:GNS} appearing in the paper by \cite{nash1958continuity}:

\begin{lemma}
  \label{lemma}
  Let $I$ be the normalised characteristic function of the unit ball
  in $\R^N$,
  \begin{equation}\label{normalised}
    I(z) := \frac{1}{\omega_N} \quad \text{if $|z| < 1$},
    \qquad
    I(z) = 0 \quad \text{otherwise,}
  \end{equation}
  where $\omega_N$ is the volume of the unit ball in dimension
  $N$. There exists a constant $C = C(N)$ depending only on $N$ such
  that
  \begin{equation}
    \label{eq:energy-ineq}
    \D^I_2(u)
    \geq
    C \, \min \left\{
      \|u\|_1^{-\frac{4}{N}} \|u\|_2^{2 + \frac{4}{N}},\ \|u\|_2^2
    \right\},
  \end{equation}
  for all $u \in L^1(\R^N) \cap  L^2(\R^N)$.
\end{lemma}

We point out that the constant $C$ can be estimated explicitly by
following the calculations in the proof below.

\begin{proof}
  Along the proof we call $C_1, C_2, \dots$ several constants that
  depend only on the dimension $N$. We will use the following
  property, which holds for some constant $C_1 > 0$:
  \begin{equation*}
    1 - \hat{I}(\xi) \geq \frac{1}{C_1} \min \{1, |\xi|^2\},
    \qquad \text{for all $\xi \in \R^N$}
  \end{equation*}
  or, in other words,
  \begin{equation}
    \label{eq:1}
    (1 - \hat{I}(\xi))^{-1} \leq C_1 \max \{1, |\xi|^{-2}\},
    \qquad \text{for all $\xi \in \R^N$}.
  \end{equation}
  Since $I$ has integral one we can write, using that the Fourier
  transform is an isometry of $L^2(\R^N; \C)$,
  \begin{equation*}
    \D_2^I(u) = 2 \ap{u,  u - I * u }
    = 2 \big\langle \hat{u}, (1-\hat{I}) \hat{u} \big\rangle
    = 2 \irN (1 - \hat{I}) |\hat{u}|^2,
  \end{equation*}
  where $\ap{\cdot, \cdot}$ denotes the usual inner product in the
  space of $L^2$ complex functions in $\R^N$. We can break the
  integral of $\|u\|_2$ in two parts, for any
  $\delta > 0$:
  \begin{equation}
    \label{eq:5}
    \|u\|_2^2 = \|\hat u\|_2^2
    = \int_{|\xi|\leq \delta}|\hat u(\xi)|^2\d\xi
    + \int_{|\xi|> \delta}|\hat u(\xi)|^2\d\xi.
  \end{equation}
  These two terms can be estimated as follows: for the first one,
  \begin{equation}
    \label{eq:6}
    \int_{|\xi|\leq \delta}|\hat u(\xi)|^2\d\xi
    \leq
    \|u\|_1^2 \int_{|\xi|\leq \delta} \d\xi
    \leq
    \omega_N \delta^N  \|u\|_1^2.
  \end{equation}
  For the second one, using \eqref{eq:1} and assuming $\delta < 1$ we
  have
  \begin{equation}
    \label{eq:7}
    \begin{split}
      \int_{|\xi|> \delta}|\hat u(\xi)|^2\d\xi &\leq C_1 \int_{|\xi|>
        \delta}\left(1-\hat I(\xi)\right) \max\{ 1, |\xi|^{-2}\}
      \,|\hat u(\xi)|^2\d\xi
      \\
      & \leq C_1 \int_{|\xi|> \delta}\left(1-\hat
        I(\xi)\right) \max\{ 1, \delta^{-2}\} \,|\hat u(\xi)|^2\d\xi
      \\
      & \leq \frac{C_1}{\delta^2} \int_{|\xi|>
        \delta}\left(1-\hat I(\xi)\right)\,|\hat u(\xi)|^2\d\xi
      \leq \frac{C_1}{\delta^2} \D_2^I(u).
    \end{split}
  \end{equation}
  Using \eqref{eq:6} and \eqref{eq:7} in \eqref{eq:5} we obtain
  \begin{equation}
    \label{eq:2}
    \|u\|_2^2 \leq
    \omega_N \delta^N  \|u\|_1^2 + \frac{C_1}{\delta^2} \D_2^I(u),
    \qquad \text{for any $0 < \delta < 1$.}
  \end{equation}
  We would like to optimise this quantity in $\delta$, but \emph{it is
    only valid for} $0 < \delta < 1$. If we could choose $\delta$
  freely we would take the one that achieves the best bound in the
  inequality \eqref{eq:2}, that is,
  \[
  \delta_0 := \left(
    \frac{2 C_1 \D_2^I(u)}{N \omega_N \|u\|_1^2}
  \right)^{\frac{1}{N+2}}.
  \]
  Now we discuss two cases: \medskip
  \\
  \textbf{Case 1.} If $\delta_0 < 1$, then replacing $\delta$ by
  $\delta_0$ in \eqref{eq:2} we have
  \begin{equation*}
    \|u\|_2^2 \leq
    \omega_N^{\frac{2}{N+2}} C_1^{\frac{N}{N+2}}
    \left(1 + \frac{N}{2} \right)
    \left(\frac{2}{N}\right)^{\frac{N}{N+2}}
    \|u\|_1^{\frac{4}{N+2}}\D_2^I(u)^{\frac{N}{N+2}}.
  \end{equation*}
  Equivalently,
  \begin{equation}
    \label{2}
    \D_2^I(u)
    \geq
    C_2 \|u\|_1^{-\frac{4}{N}} \|u\|_2^{2 + \frac{4}{N}}
  \end{equation}
  where
  $C_2 = \omega_N^{-\frac{2}{N}} C_1^{-1} \left(1 + \frac{N}{2}
  \right)^{-\frac{N+2}{N}} \frac{N}{2}$.
  \medskip
  \\
  \textbf{Case 2.} If $\delta_0 \geq 1$ then this means that
  \begin{align*}
    N \omega_N \|u\|_1^2 \leq 2 C_1 \D_2^I(u).
  \end{align*}
  In this case, choosing $\delta=1$ in \eqref{eq:2} and using the
  above inequality we get
  \begin{equation*}
    \|u\|_2^2
    \leq \omega_N \|u\|_1^2 + C_1 \,\D_2^I(u)
    \leq \left(1 + \frac{2}{N} \right) C_1 \, \D_2^I(u),
  \end{equation*}
  or
  \begin{equation}
    \label{3}
    \D_2^I(u) \geq C_3 \|u\|_2^2
  \end{equation}
  with $C_3 := C_1^{-1} \left(1 + \frac{2}{N} \right)^{-1}$.
  
  \medskip
  Finally, summarising \eqref{2} and \eqref{3} we obtain
  \begin{equation}
    \label{eq:4}
    \D_2^I(u) \geq
    C_4 \, \min \left\{
      \|u\|_1^{-\frac{4}{N}} \|u\|_2^{2 + \frac{4}{N}},
      \,
      \|u\|_2^2
    \right\}
  \end{equation}
  with $C_4 := \max \left\{ C_2, C_3 \right\}$. This proves
  \eqref{eq:energy-ineq} with $C = C_4$.
\end{proof}

Notice that $\D_2^J(u)$ satisfies the following scaling property. For
$\lambda > 0$ and any function $g$ on $\R^N$ we
denote
\begin{equation*}
  g_\lambda(z) := g(z / \lambda),
  \qquad z \in \R^N.
\end{equation*}
Then one sees that
\begin{equation}
  \label{eq:3}
  \D_2^{J_\lambda}(u) = \lambda^{2N} \D_2^J(u_{\frac{1}{\lambda}}).
\end{equation}
This easily gives the following extension of Lemma \ref{lemma}:
\begin{corollary}[$L^2$ energy inequality]
  \label{cor:energy-ineq}
  Let $J$ satisfy Hypothesis \ref{hyp:J}. There is some constant
  $C = C(N)$ that depends only on the dimension $N$ such that
  \begin{equation}
    \label{eq:energy-ineq-2}
    \D_2^J(u) \geq
    C r
    \min \left\{
      R^{N + 2} \|u \|_1^{-\frac{4}{N}} \|u\|_2^{2 + \frac{4}{N}},\
      R^{N} \|u \|_2^2
    \right\}
  \end{equation}
  for all $u \in L^2(\R^N) \cap L^1(\R^N)$.
\end{corollary}

\begin{proof}[Proof of Corollary \ref{cor:energy-ineq}]
  Call $I = I(z)$ the normalised characteristic of the unit ball, and define
  \begin{equation*}
    K(z) := \frac{1}{r\, \omega_N} J(R z),
    \qquad z \in \R^N.
  \end{equation*}
  Then
  \begin{equation*}
    K(z) \geq I(z)
    \qquad \text{for all $z \in \R^N$}
  \end{equation*}
  so
  \begin{equation*}
    \D^K_2(u) \geq \D^I_2(u).
  \end{equation*}
  Since $J = r \omega_N K_R$, due to the scaling \eqref{eq:3}
  we have
  \begin{equation*}
    \D^J_2(u) = r\, \omega_N R^{2N} \D_2^K (u_{\frac{1}{R}})
    \geq  r\, \omega_N R^{2N} \D_2^I (u_{\frac{1}{R}}).
  \end{equation*}
  Hence we can use Lemma \ref{lemma} (writing $C_N$ to denote the
  constant $C$ in it) to say that
  \begin{align*}
    \D_2^J(u)
    &\geq
    r \omega_N R^{2N}
    C_N \, \min \left\{
    \|u_{\frac{1}{R}}\|_1^{-\frac{4}{N}} \|u_{\frac{1}{R}}\|_2^{2 + \frac{4}{N}},\ \|u_{\frac{1}{R}}\|_2^2
    \right\}
    \\
    &=
    r \omega_N
    C_N \, \min \left\{
      R^{N + 2} \|u \|_1^{-\frac{4}{N}} \|u\|_2^{2 + \frac{4}{N}},\
      R^{N} \|u \|_2^2
    \right\}.
  \end{align*}
  This shows the result.
\end{proof}

\medskip Corollary \ref{cor:energy-ineq} gives the case $p=2$ of
Theorem \ref{teo-principal}. In order to obtain the general case for
$p \geq 2$ and complete the proof, let us first state a simple
elementary without proof inequality in the next lemma:

\medskip

\begin{lemma}
  Let $p>1$, there exists $c(p)>0$ such that
  \begin{equation}\label{B}
    (a-b)(a^{p-1}-b^{p-1})\geq c(p)\,(a^{p/2}-b^{p/2})^2,
    \hspace{0,4cm} \textrm{ for all }a,b\geq 0.
  \end{equation}
\end{lemma}


\medskip
We can now complete the proof of Theorem \ref{teo-principal}:

\begin{proof}[Proof of Theorem \ref{teo-principal}]
  As explained at the beginning of Section \ref{sec:inequalities}, we
  may assume that $u$ is nonnegative. By using the inequality 
  \eqref{B} we obtain
  \begin{align*}
    \D_p^J(u) & = \irN \irN J(x-y)(u(x)-u(y))(u(x)^{p-1}-u(y)^{p-1})dxdy
    \\
              & \geq c(p) \irN \irN J(x-y)(u(x)^{p/2}-u(y)^{p/2})^2dxdy
    \\
              & = c(p) \D_2^J(u^{p/2}).
  \end{align*}
  Now, by virtue of Corollary \ref{cor:energy-ineq}, and calling $C_N$
  the constant in it, it follows that
  \begin{align*}
    \D_p^J(u)
    & \geq c(p) C_N r\,
      \min \left\{
      R^{N+2} \|u^{p/2}\|_1^{-\frac{4}{N}}\, \|u^{p/2}\|_2^{2+\frac{4}{N}},\,
      R^N \|u^{p/2}\|_2^2
      \right\}
    \\
    & = c(p) C_N r\, \min \left\{
      R^{N+2} \|u\|_{\frac{p}{2}}^{-\frac{2p}{N}}\, \|u\|_p^{p\left(1+\frac{2}{N} \right)},\,
      R^N \|u\|_p^p
      \right\}.
  \end{align*}
  Finally, due to the interpolation formula
  \begin{equation*}
    \|u\|_{\frac{p}{2}}\leq \|u\|_1^{\frac{1}{p-1}} \, \|u\|_p^{\frac{p-2}{p-1}}
  \end{equation*}
  (note that $p \geq 2$ is used here) we conclude that
  \begin{equation*}
    \D_p^J(u) \geq c(p) C_N\, r\,
    \min \left\{
      R^{N+2} \|u\|_1^{-p\,\gamma}\, \|u\|_p^{p(1+\gamma)},\,
      R^N \|u\|_p^p  \right\},
  \end{equation*}
  with $\gamma = \frac{2}{N(p-1)}$.
\end{proof}

\section{Energy inequalities involving derivatives}
\label{sec:ineq-deriv}

We now prove Theorem \ref{k derivative estimate}, an inequality which
is useful when studying the decay of derivatives of solutions to
nonlocal diffusion equations:

\begin{proof}[Proof of Theorem \ref{k derivative estimate}]
  The proof is a direct extension of the technique in the proof of
  Theorem \ref{teo-principal}. We follow the same steps. First, we
  assume that $J$ is the normalised characteristic function of the
  unit ball in $\RR^N$, given by \eqref{normalised}. Then, closely
  following Lemma \ref{lemma}, we claim
  \begin{equation}
    \label{ineq-k-normal}
    \D_2^J(D^k u) \geq
    C_N \min \left\{ \|u\|_1^{-\frac{4}{N+2k}}\, \|D^k u\|_2^{2+\frac{4}{N+2k}},
      \,\, \|D^k u\|_2^2
    \right\}
  \end{equation}
  for some constant $C_N > 0$ depending only on $N$. As in the proof
  of Lemma \ref{lemma},
  \begin{equation*}
    \D_2^J(D^k u) = 2 \irN (1 - \hat{J}) |\widehat{D^k u}|^2.
  \end{equation*}
  Now, recalling inequality \eqref{eq:1} and taking into account that
  $| \widehat{D^k u}(\xi)|^2 =|\xi|^{2k} |\hat{u}(\xi)|^2 \leq
  |\xi|^{2k} \|u\|_1^2$ we obtain for $0<\delta \leq 1$ that
  \begin{align}
    \label{delta gradient}
    \nonumber \|D^k u\|_2^2 =
    \|\widehat{D^k u}\|_2^2
    & = \int_{|\xi|\leq \delta} |\widehat{D^k u}(\xi)|^2 \d\xi
      + \int_{|\xi| > \delta} |\widehat{D^k u}(\xi)|^2 \d\xi
    \\  
    & \leq \|u\|_1^2 \int_{|\xi|\leq \delta} |\xi|^{2k} \d\xi
      + \frac{C_1}{\delta^2} \int_{|\xi| > \delta} ({1-\hat{J}(\xi)})
      |\widehat{D^k u}(\xi)|^2 \d\xi
    \\
 \nonumber   & \leq \omega_N \delta^{2k + N} \, \|u\|_1^2
      + \frac{C_1}{2 \delta^2}\,\D_2^J(D^k u).
  \end{align}
  Choose
  \[
  \delta_0=\left( \frac{2\,C_1\,\D_2^J(D^ku)}{(N+2k)\,\|u\|_1^2\,\omega_N}   \right)^{\frac{1}{N+2k+2}}.
  \]
  We obtain, as in Lemma \ref{lemma}, two possibilities: if
  $\delta_0 \leq 1$, we get
  \begin{equation}\label{k delta less than one}
    \|D^ku\|_2^2 \leq C_2\,\|u\|_1^{2\mu_k}\, \D_2^J(D^ku)^{1-\mu_k}
  \end{equation}
  with $\mu_k=\frac{2}{N+2+2k}$ and
  $C_2=\left(\frac{N}{2}+k \right)^{\mu_k}\left(1+\frac{2}{N+2k}
  \right)^{1-\mu_k}C_1\omega_N^{\mu_k}$.
  In the other case, $\delta_0>1$, we get
  \begin{equation}\label{k delta greater than one}
    \|D^ku\|_2^2 \leq \frac{2C_1}{(N+2k)\mu_k}\, \D_2^J(D^ku).
  \end{equation}
  Collecting inequalities \eqref{k delta less than one} and \eqref{k
    delta greater than one} we have
  \[
  \|D^ku\|_2^2
  \leq
  C_N \max \left \{
    \|u\|_1^{2\mu_k}\,\D_2^J(D^ku)^{1-\mu_k},\,
    \D_2^J(D^ku)
  \right \}
  \]
  where $C_N=\max \left \{C_2, \frac{2C_1}{(N+2k)\mu_k} \right
  \}$.
  Reversing the inequality we have thus proved \eqref{ineq-k-normal}.

  \medskip In order to complete the proof we consider any $J$
  satisfying Hypothesis \ref{hyp:J}. We have a scaling property which
  is an extension of \eqref{eq:3}:
  \begin{equation}
    \label{eq:scaling-Dk}
    \D_2^{J_\lambda} (D^k u) = \lambda^{2N-k} \D_2^J (D^k u_{1/\lambda}),
  \end{equation}
  for any $\lambda > 0$. Of course, we also have
  $D^k u_\lambda = \lambda^{-k} (D^k u)_\lambda$, the usual scaling
  for derivatives. If $I$ denotes the characteristic function of the
  unit ball on $\R^N$ and we define
  $K = \frac{1}{r \omega_N} J_{1/R}$ as in the proof of
  Corollary \ref{cor:energy-ineq} then $K \geq I$, and
  $J = r \omega_N K_R$. Using the scaling property
  \eqref{eq:scaling-Dk} and the normalised case \eqref{ineq-k-normal}
  we see that
  \begin{multline*}
    \D_2^J (D^k u) = r \omega_N R^{2N-k} \D_2^K (u_{1/R})
    \geq
    r \omega_N R^{2N-k} \D_2^I (u_{1/R})
    \\
    \geq
    r \omega_N R^{2N-k}
    C_N \min \left\{ \|u_{1/R}\|_1^{-\frac{4}{N+2k}}\, \|D^k u_{1/R}\|_2^{2+\frac{4}{N+2k}},
      \,\, \|D^k u_{1/R}\|_2^2
    \right\}
    \\
    =
    r \omega_N R^{2N-k}
    C_N \min \left\{ R^{2k - N + 2} \|u\|_1^{-\frac{4}{N+2k}}\, \|D^k u\|_2^{2+\frac{4}{N+2k}},
      \,\, R^{2k - N} \|D^k u\|_2^2
    \right\}
    \\
    =
    r \omega_N
    C_N \min \left\{ R^{k + N + 2} \|u\|_1^{-\frac{4}{N+2k}}\, \|D^k u\|_2^{2+\frac{4}{N+2k}},
      \,\, R^{k + N} \|D^k u \|_2^2
    \right\},
  \end{multline*}
  which shows the result.
\end{proof}

We point out that analogous results can be stated for other
differential operators. As an example we consider $\nabla
u$. Following the notation of the preceding section we set
\begin{equation}\label{gradient}
  \D_2^J(\nabla u) =
  \irN \irN J(x-y)\left|\nabla u(x) -\nabla u(y)  \right|^2 \dx \dy,
\end{equation}
defined for any $u \in H^1(\R^N)$. Reasoning along the same lines as
in the previous result one obtains the following result for $\nabla u$
(notice that this is not the same as the $k=1$ case of Theorem \ref{k
  derivative estimate}, since $D^1 u$ is not equal to $\nabla u$):
\begin{theorem}
  \label{gradient estimate}
  Let $N \geq 1$ be an integer and $J:\R^N \to \R$ be a function
  satisfying Hypothesis \ref{hyp:J}. There exists a positive constant
  $C=C(N)$ such that
  \begin{equation}
    \label{ineq-gradient}
    \D_2^J(\nabla u) \geq
    C r \min \left\{ R^{N+3} \|u\|_1^{-\frac{4}{N+2}}\,
      \|\nabla u\|_2^{2+\frac{4}{N+2}},
      \,\, R^{N+1} \|\nabla u\|_2^2
    \right\},
  \end{equation}
  for all $u \in H^1(\R^N) \cap L^1(\R^N)$.
\end{theorem}

\begin{proof}
  If $J$ has integral one we can write, as before,
  \begin{equation*}
    \D_2^J(\nabla u) = 2 \ap{\nabla u,  \nabla u - J * \nabla u }
    = 2 \big\langle \widehat{\nabla u}, (1-\hat{J}) \widehat{\nabla u} \big\rangle
    = 2 \irN (1 - \hat{J}) |\widehat{\nabla u}|^2.
  \end{equation*}
  Since
  $| \widehat{\nabla u}(\xi)|^2 =|\xi|^2|\hat{u}(\xi)|^2\leq
  |\xi|^2\|u\|_1^2$,
  one can follow the same reasoning as in the $k=1$ case of Theorem
  \ref{k derivative estimate} to obtain the result.
\end{proof}

\section{Some applications}
\label{sec:applications}

\subsection{The linear nonlocal diffusion equation in convolution
  form}
\label{sec:linear}

The most direct application of the inequalities in the previous
section concerns the long-time behaviour of the linear nonlocal
diffusion equation:
\begin{equation}
  \label{eq:nonlocal2}
  \p_t u(t,x) = \irN J(x-y) (u(t,y) - u(t,x)) \d y,
\end{equation}
where $t \geq 0$ is the time variable, $x \in \RR^N$ is the space
variable, $u = u(t,x) \in \R$ is the unknown, and $J$ is the
\emph{diffusion kernel}. As a straightforward consequence of Theorem
\ref{teo-principal} we obtain Theorem \ref{thm:decay-main}, which we
prove now.

\begin{proof}[Proof of Theorem \ref{thm:decay-main}]
  The regularity of the solution $u$ allows us to write the following
  $H$-theorem for the $L^p$ norm:
  \begin{equation}\label{f}
    \ddt \|u\|_p^p = - \D_p^J(u).
  \end{equation}
  Due to Theorem \ref{teo-principal}, and taking into account that
  $\|u(t,\cdot)\|_1 = \|u_0\|_1$ (mass conservation), we have
  \begin{equation*}
    \ddt \|u\|_p^p \leq - C r \min\left\{
      R^{N+2} \|u_0\|_1^{-p \gamma}\, \|u\|_p^{p(1+\gamma)},\,
      R^N \|u\|_p^p
    \right\},
  \end{equation*}
  for some constant $C = C(N,p)$. This is a differential inequality
  for $\|u\|_p^p$ which allows us to compare it to the solution to the
  equation
  \begin{equation*}
    X'(t) = - C r \min\left\{
      R^{N+2} \|u_0\|_1^{-p \gamma}\, X(t)^{(1+\gamma)},\,
      R^N X(t)
    \right\}.
  \end{equation*}
  We can then apply Lemma \ref{lem:ode} with
  \begin{equation*}
    C_1 := C r R^{N+2} \|u_0\|_1^{-p \gamma},
    \qquad
    C_2 := C r R^N,
  \end{equation*}
  to obtain the result.
\end{proof}

\begin{lemma}
  \label{lem:ode}
  Take $C_1, C_2, \gamma > 0$ and let $X = X(t)$ be a solution on
  $[0,+\infty)$ to the ordinary differential equation
  \begin{equation}
    \label{eq:ode}
    X'(t) = -\min\left\{
      C_1 X(t)^{1+\gamma},
      C_2 X(t)
    \right\}.
  \end{equation}
  with $X(0) > 0$. Then we have
  \begin{equation}
    \label{eq:X-bound}
    X(t) \leq
    \begin{cases}
      X(0)
      &\qquad
      \text{for $t \in [0,t_0]$,}
      \\
      \big( X(0)^{-\gamma}
        + \gamma C_1 (t-t_0)
      \big)^{-\frac{1}{\gamma}}
      &\qquad \text{for $t \in (t_0,+\infty)$}
    \end{cases}
  \end{equation}
  where
  \begin{equation*}
    t_0 = \max\left\{
      0,\
      \frac{1}{C_2}
      \log \big( C_2^{-\frac{1}{\gamma}} C_1^{\frac{1}{\gamma}} X(0) \big)
    \right\}.
  \end{equation*}
\end{lemma}

\begin{remark}
  \label{rem:X}
  The solution of the ordinary differential equation in the above
  lemma is actually explicit (see the proof), and we just aim to give
  a simple statement that captures the decay of the solution as
  $t \to +\infty$. One can simplify even further and say that there is
  a constant $C = C(C_1, C_2, \gamma, X(0))$ such that
  \begin{equation*}
    X(t) \leq C (1 + t)^{-\frac{1}{\gamma}},
    \qquad
    \text{for all $t \geq 0$.}
  \end{equation*}
  This is easily deduced from \eqref{eq:X-bound} with
  \begin{equation*}
    C := \sup_{t \geq 0} \frac{X(t)}{(1+t)^{-\frac{1}{\gamma}}},
  \end{equation*}
  which is finite since both $X$ and $(1+t)^{-\frac{1}{\gamma}}$ have
  the same decay as $t \to +\infty$, and obviously depends only on
  $C_1, C_2, \gamma$ and $X(0)$.
\end{remark}

\begin{proof}[Proof of Lemma \ref{lem:ode}]
  By usual theorems in ordinary differential equations, equation
  \eqref{eq:ode} has a unique solution on $[0,+\infty)$ with the given
  initial condition $X(0)$, and this solution is nonnegative on
  $[0,+\infty)$. The condition that decides which of the two terms
  achieves the minimum at each time $t$ is whether
  \begin{equation}
    \label{eq:8}
    X(t)^\gamma \leq \frac{C_2}{C_1}
  \end{equation}
  or not. Since $X$ is nonincreasing, once this condition is satisfied at a
  certain $t_0 \geq 0$ it will be satisfied for all $t \geq t_0$. With
  this it is easy to calculate the explicit solution, given by
  \begin{equation*}
    X(t) =
    \begin{cases}
      X(0) e^{-C_2 t}
      &\qquad
      \text{for $t \in [0,t_0]$,}
      \\
      \big( X(t_0)^{-\gamma}
        + \gamma C_1 (t-t_0)
      \big)^{-\frac{1}{\gamma}}
      &\qquad \text{for $t \in (t_0,+\infty)$}
    \end{cases}
  \end{equation*}
  where
  \begin{equation*}
    t_0 = \max\left\{
      0,\
      \frac{1}{C_2}
      \log \big( C_2^{-\frac{1}{\gamma}} C_1^{\frac{1}{\gamma}} X(0) \big)
    \right\}.
  \end{equation*}
  One obtains the result by noticing that $X(0) e^{-C_2 t} \leq X(0)$
  and $X(t_0) \leq X(0)$.
\end{proof}

\medskip Similarly, with the help of the previous lemma the
inequalities in Theorem \ref{k derivative estimate} imply the decay in
Theorem \ref{thm:decay-deriv}:

\begin{proof}[Proof of Theorem \ref{thm:decay-deriv}]
  If $u$ satisfies equation \eqref{eq:nonlocal} then $D^k u$ satisfies
  the same equation, with initial condition $D^k u (0,x) = D^k
  u_0(x)$. Hence we have, as in \eqref{eq:nonlocal-Lp-evol},
  \begin{equation*}
    \ddt \| D^k u \|_2^2 = -\D_2^J(D^k u).
  \end{equation*}
  Using Theorem \ref{k derivative estimate} we obtain
  \begin{equation*}
    \ddt \| D^k u \|_2^2 \leq
    - C r \min \left\{ R^{k + N + 2} \|u\|_1^{-\frac{4}{N+2k}}\,
      \|D^k u\|_2^{2+\frac{4}{N+2k}},
      \,\, R^{k + N} \|D^k u\|_2^2
    \right\}.
  \end{equation*}
  This is again a differential inequality for $\|D^k u\|_2^2$, to
  which we can apply Lemma \ref{lem:ode} with
  \begin{equation*}
    C_1 = C r R^{k+N+2} \|u_0\|_1^{-\frac{4}{N+2k}},
    \qquad
    C_2 = C r R^{k+N}.
  \end{equation*}
  This directly gives the result.
\end{proof}

\subsection{General linear mass-conserving nonlocal equations}
\label{no-convolution}

In this section we prove Theorem \ref{thm:decay-general}, which
concerns equation \eqref{general-kernel}, recalled here:
\begin{equation}
  \label{general-kernel-2}
  \p_t u(t,x) = \irN K(x,y) u(t,y) \d y - \sigma(x) u(t,x),
\end{equation}
where $K \: \RR^N \times \RR^N \to [0,\infty)$ is a general kernel
(not necessarily symmetric) and $\sigma \: \R^N \to [0,+\infty)$ is a
function.
In order to apply our strategy to equation \eqref{general-kernel-2} we
need to have suitable Lyapunov functionals for it. To our knowledge,
the most general setting in which one can do this is that of the
so-called \emph{general relative entropy} method
\citep{MMP,MMP2004-general}, which we state here in a particular case:
assume that \eqref{eq:mass-conservation} holds and that
\begin{equation}
  \label{eq:exists_equilibrium}
  \text{There exists a positive equilibrium $u_\infty \: \R^N \to (0,+\infty)$ of \eqref{general-kernel-2}}.
\end{equation}
(That is, a solution $u_\infty$ of \eqref{general-kernel-2} which does
not depend on time $t$.) Then it is known that
\begin{equation*}
  \ddt \int_{\R^N} \Phi\left( \frac{u(t,x)}{u_\infty(x)} \right)
  u_\infty(x) \d x \leq 0,
\end{equation*}
whenever $\Phi$ is a convex function and $u$ is any solution of
\eqref{general-kernel-2}. This fact is well-known in probability theory
(see the review by \cite{Chafai04}) and is a direct consequence of the
general relative entropy method \citep{MMP2004-general}. The explicit
form of its time derivative can be found in \cite{MMP2004-general}:
\begin{multline}
  \label{eq:general-entropy-evol}
  \ddt \int_{\R^N} \Phi\left( f(x)\right) u_\infty(x) \d x
  \\
  = - \int_{\R^N} \int_{\R^N}
  K(x,y) u_\infty(y) \big(
    \Phi'(f(x)) (f(x) - f(y)) - \Phi(f(x)) + \Phi(f(y))
  \big)
  \d x \d y,
\end{multline}
where we denote $f(t,x) \equiv u(t,x) / u_\infty(x)$, and where the
$t$ variable has been omitted for shortness. Notice that the integrand
is always nonnegative due to the convexity of $\Phi$. The following
particular cases are of interest for us here: for $\Phi(f) = |f|^p$
with $p > 1$ we have
\begin{equation}
  \label{eq:general-Lp-evol}
  \ddt \|u\|_p^p = - \mathcal{E}_p^K(f),
\end{equation}
where the \emph{dissipation} $\mathcal{E}_p^K(f)$ is an operator
acting only on the $x$ variable. Its expression is given by the right
hand side of \eqref{eq:general-entropy-evol} (with $\Phi(f) = |f|^p$)
and is not so simple. But if we additionally assume that
\begin{equation}
  \label{eq:detailed-balance}
  K(x,y) u_\infty(y) = K(y,x) u_\infty(x),
  \qquad \text{for all $x,y \in \R^N$},
\end{equation}
then one can check that
\begin{multline}
  \label{eq:Lp-dissipation-sym}
  \mathcal{E}_p^K(f)
  \\
  = p \int_{\R^N} \int_{\R^N}
  K(x,y) u_\infty(y) \big(
    (f(x))^{p-1} (f(x) - f(y)) - (f(x))^p + (f(y))^p
  \big)
  \d x \d y
  \\
  =
  \frac{p}{2} \int_{\R^N} \int_{\R^N} \big( f(x)
    - f(y) \big)
   \left( f(x)^{p-1}
    - f(y)^{p-1} \right)
  K(x,y) u_\infty(y) \d x \d y
\end{multline}
for all nonnegative functions $f$; note the parallel with
\eqref{eq:nonlocal-Lp-evol}. The last equality in
\eqref{eq:Lp-dissipation-sym} is obtained by noticing that the
integrals corresponding to $f(x)^p$ and $f(y)^p$ cancel out (easily
seen by using \eqref{eq:detailed-balance}), and using
\eqref{eq:detailed-balance} again to symmetrise the remaining
integral:
\begin{multline*}
   \int_{\R^N} \int_{\R^N}
  K(x,y) u_\infty(y)
    f(x)^{p-1} (f(x) - f(y))
  \d x \d y
  \\
  =
  \frac{1}{2} \int_{\R^N} \int_{\R^N} K(x,y) u_\infty(y)
   \left( f(x)^{p-1} - f(y)^{p-1} \right)
  \big( f(x) - f(y) \big)
  \d x \d y.
\end{multline*}
Condition \eqref{eq:detailed-balance} is
known in probability as the \emph{detailed balance} or
\emph{reversibility} condition (it holds for example if
$u_\infty \equiv 1$ and $K$ is symmetric). If one works in a setting
where \eqref{eq:general-Lp-evol} holds then it may still be possible
to use the inequality in Theorem \ref{teo-principal} (or related ones)
and deduce some information on the rate of decay of solutions.

\begin{proof}[Proof of Theorem \ref{thm:decay-general}]
  Condition \eqref{eq:K-existence} is easily seen to imply that the
  linear operator given by
  \begin{equation*}
    L u(x) = \irN K(x,y) u(y) \d y - \sigma(x) u(x),
    \qquad x \in \R^N,
  \end{equation*}
  is well defined and bounded both in $L^1(\R^N)$ and
  $L^p(\R^N)$. This shows that equation \eqref{general-kernel-2} with
  initial condition $u_0$ has a unique solution in
  $\mathcal{C}^1([0,+\infty), L^p(\R^N) \cap L^1(\R^N))$ which
  conserves mass (that is,
  $\int_{\R^N} u(t,x) \d x = \int_{\R^N} u_0(x) \d x$ for all
  $t \geq 0$), and that it satisfies the entropy property
  \eqref{eq:nonlocal-Lp-evol}. It is also seen easily that equation
  \eqref{general-kernel-2} preserves sign: if the initial condition is
  nonnegative (nonpositive) then $u(t,x)$ is nonnegative (nonpositive)
  for all $t,x$. As a consequence, it is enough to prove the result
  when $u_0$ is nonnegative --- the general result is then obtained by
  linearity from $u_0 = u_0^+ - u_0^-$, with $u_0^+ := \max\{u_0, 0\}$
  and to $u_0^- := \max\{-u_0, 0\}$.

  For $x, y \in \R^N$ call
  \begin{equation*}
    \tilde{K}(x,y) := r, \quad \text{if $|x-y| \leq R$},
    \qquad
     \tilde{K}(x,y) := 0 \quad \text{otherwise}
  \end{equation*}
  and
  \begin{equation*}
    J(x) := r, \quad \text{if $|x| \leq R$},
    \qquad
    J(x) := 0 \quad \text{otherwise}.
  \end{equation*}
  Due to Hypothesis \ref{eq:hyp1-general} and \eqref{eq:equilibrium-bounded} we
  have
  \begin{equation*}
    K(x,y) u_\infty(y) \geq \frac{1}{m} \tilde{K}(x,y).
  \end{equation*}
  Hence, since $\tilde{K}$ is symmetric, using the same symmetrisation
  trick as in \eqref{eq:Lp-dissipation-sym},
  \begin{align*}
    \mathcal{E}_p^K(f) & \geq \mathcal{E}_p^{\tilde{K}}(f)
    \\
   & \geq
    \frac{p}{2 m} \int_{\R^N}
    \big(
    (f(x))^{p-1} (f(x) - f(y)) - (f(x))^p + (f(y))^p
    \big)
    \tilde{K}(x,y) \d x \d y
    \\
    &=
    \frac{p}{2 m} \int_{\R^N}
    \left( f(x)^{p-1} - f(y)^{p-1} \right)
    \big( f(x) - f(y) \big)
    \tilde{K}(x,y) \d x \d y
   \\
    &= \mathcal{D}_p^J(f)
  \end{align*}
  for any nonnegative function $f$, where $\mathcal{D}_p^J(f)$ is the
  dissipation in \eqref{eq:Dp2}. Hence for the (nonnegative) solution
  $u$, using Theorem \ref{teo-principal} and calling
  \begin{equation*}
    X(t) := \int_{\R^N} f^p u_\infty
    = \int_{\R^N} \left( \frac{u(t,x)}{u_\infty(x)} \right)^p u_\infty(x) \d x
  \end{equation*}
  we have
  \begin{align*}
    \ddt X(t)
    &= -\mathcal{E}_p^K(f)
    \\
    &\leq
    - \mathcal{D}_p^J(f)
    \\
    & \leq - C \min \{\|f\|_1^{-p \gamma} \|f\|_p^{p(1+\gamma)},
    \|f\|_p^p \}
    \\
    & \leq
    - C_2 \min \{\|u_0\|_1^{-p \gamma} X(t)^{1+\gamma},
    X(t) \},
  \end{align*}
  where $C_2$ also depends on $m$, and we have used mass conservation
  and again the bounds in \eqref{eq:equilibrium-bounded}. Due to the
  differential inequality in Lemma \ref{lem:ode} we obtain that
  \begin{equation*}
    X(t) \leq C (1+t)^{-\frac{N(p-1)}{2}},
     \qquad \text{for all $t \geq 0$},
   \end{equation*}
   for some constant $C$ as stated in the result. We complete the
   proof by noticing that
   \begin{equation*}
     \|u\|_p^p
     \leq
     m^{1-p} \int_{\R^N} \left( \frac{u(t,x)}{u_\infty(x)} \right)^p u_\infty(x) \d x
     = m^{1-p} X(t).
     \qedhere
   \end{equation*}
\end{proof}

\subsection{A nonlocal dispersal equation}
\label{sec:dispersal}

We consider the following integro-differential equation (the dispersal
model that was briefly mentioned in the introduction):
\begin{equation}
  \label{dispersal-eq}
  \p_t u(t,x) =
  \ir J\left(\frac{x-y}{g(y)}\right)
  \frac{u(t,y)}{g(y)}\d y-u(t,x),
  \qquad
  \text{in $\RR \times [0,\infty)$,}
\end{equation}
with a prescribed initial data $u(x,0)=u_0(x)$ defined on $\RR$. Here
$J$ is an even, positive, smooth function such that
$\int_{\RR} J(x)\d x = 1$ and $\hbox{supp } J = [-1,1]$, and $g$ is a
continuous positive function which accounts for the {\it {dispersal
    distance}} which depends on the departing point. In this model $u$
represents the spatial distribution of a certain species, and $g$
models the heterogeneity of the environment which can affect the
distribution of a species through space-dependent dispersal
strategies. This model was proposed in \cite{CCEM} (see also
\cite{CEG-MM2011,CEG-MM2016,CEG-MM}). It was shown there that if we
assume $g$ is bounded above and below then there exists a positive
steady state solution of \eqref{dispersal-eq}, that is, a solution of
the corresponding stationary problem,
\begin{equation*}
  u_\infty(x)=
  \ir J\left(\frac{x-y}{g(y)}\right)
  \frac{u_\infty(y)}{g(y)} \d y,
  \qquad \text{in $\RR$.}
\end{equation*}
Moreover, $u_\infty$ is bounded above and below by positive
constants. It was also proved in \cite{CCEM} that any solution $u$ of
\eqref{dispersal-eq} converges to $0$ locally as $t \to \infty$. Using
the general result in Theorem \ref{thm:decay-general} we are able to
improve this asymptotic behavior obtaining a precise decay rate of the
$L^p$ norms of $u$:

\begin{theorem}
  \label{th-dispersal}
  Take $p \in [2,+\infty)$. Let $u$ be a solution of
  \eqref{dispersal-eq} with initial data
  $u_0 \in L^1(\R) \cap L^p(\R)$, and assume that
  \begin{enumerate}
  \item $J \in L^\infty(\R)$ is a bounded, nonnegative function with
    compact support, satisfying Hypothesis \ref{hyp:J},
  \item and $g$ is a continuous function satisfying
    \begin{equation*}
      \frac{1}{M} \leq g(x) \leq M,
      \qquad \text{for all $x \in \R$}
    \end{equation*}
  and  for some $M > 0$.
  \end{enumerate}
  Then for some constant $C > 0$ depending on $J$, $M$, $p$,
  $\|u_0\|_1$ and $\|u_0\|_p$,
  \begin{equation*}
    \|u\|_p^p \leq C (1+t)^{-\frac{p-1}{2}},
    \qquad \text{for all $t \geq 0$}.
  \end{equation*}
\end{theorem}

\begin{proof}
  Equation \eqref{dispersal-eq} is of the form \eqref{general-kernel-2}
  with $\sigma(x) = 1$ for all $x \in \R$ and
  \begin{equation*}
    K(x,y) = J\left( \frac{x-y}{g(y)} \right)\frac{1}{g(y)},
    \qquad
    \text{for $x,y \in \R$.}
  \end{equation*}
  Defining $K$ and $\sigma$ in this way, \eqref{eq:mass-conservation}
  is satisfied and one can check that this kernel $K$ satisfies
  Hypothesis \ref{eq:hyp1-general} and \eqref{eq:K-existence}.
  By the results in \cite{CCEM} we know that there exists an
  equilibrium $u_\infty$ satisfying \eqref{eq:equilibrium-bounded}
  (with $m$ depending only on the parameters of the problem), so we
  are in condition to apply Theorem \ref{thm:decay-general} and obtain
  the result.
\end{proof}

\begin{remark}
  One can pose equation \eqref{dispersal-eq} in $\R^N$ instead of
  $\R$. The only reason in Theorem \ref{th-dispersal} why we need the
  dimension $N$ to be $1$ is that we use the results in \cite{CCEM} to
  ensure there is a positive equilibrium $u_\infty$ which is bounded
  above and below. Theorem \ref{th-dispersal} is still true in
  dimension $N$ provided the existence of an equilibrium satisfying
  \eqref{eq:equilibrium-bounded} (with the same proof). Such existence
  of a bounded $u_\infty$ is to our knowledge an open problem in
  dimension $N > 1$.
\end{remark}


\subsection{Nonlocal diffusions with a nonlinear source}
\label{sec:with_f}

With very little change in our arguments we can obtain the same decay
estimates if we add a nonlinear source to equation
\eqref{general-kernel-2}, as long as the nonlinear source ``decreases
energy''. We consider
\begin{equation}
  \label{eq:nonlinear-source}
  \p_t u(t,x) =
  \irN K(x,y) u(t,y) \d y - \sigma(x) u(t,x)
  + f(u(t,x)).
\end{equation}
with $K$ and $\sigma$ as in Section \ref{no-convolution} and $f$ a
locally Lipschitz function satisfying the sign condition
\begin{equation}
  \label{eq:f-sign}
  f(s) s \leq 0,
  \qquad \text{for $s \in \R$.}
\end{equation}
With the same arguments as before we obtain the following:
\begin{theorem}
  \label{thm:nonlinear-source}
  Take $p \in [2,+\infty)$ and let $u$ be a solution of
  \eqref{eq:nonlinear-source} with nonnegative initial data
  $u_0 \in L^1(\R) \cap L^p(\R)$, and assume that $K$ and $\sigma$
  satisfy the conditions of Theorem \ref{thm:decay-general}. Assume
  that $f$ is a locally Lipschitz function satisfying
  \eqref{eq:f-sign}. Then for some constant $C > 0$ depending only on
  $K$, $N$, $\|u_0\|_1$ and $\|u_0\|_p$,
  \begin{equation*}
    \|u\|_p^p \leq C (1+t)^{-\frac{N}{2}},
    \qquad \text{for all $t \geq 0$}.
  \end{equation*}
\end{theorem}

\begin{proof}
  The conditions on $f$, $K$ and $\sigma$ ensure that there exists a
  solution of the equation, and that one may differentiate it in time
  to obtain the usual expression for the time derivative of
  $\|u\|_p^p$. Dropping the nonpositive term
  $f(u(t,x))\, u(t,x)\, |u(t,x)|^{p-2}$ we obtain the inequality
  \[
  \ddt \|u\|_p^p \leq - \mathcal{E}_p^K(u),
  \]
  Arguing as in the proof of Theorem \ref{thm:decay-general} we obtain
  the asymptotic decay. Observe that the total mass of the solution is
  nonincreasing, since $f(s)\leq 0$ for $s\geq 0$.
\end{proof}

This equation was treated in \cite{AMRT2010,citeulike:5333882} where a
restriction on the dimension ($N\geq 3$) and $K$ symmetric are required in order to
establish the asymptotic behavior.

\section*{Acknowledgements}

J.~A.~Cañizo was supported by the Spanish \emph{Ministerio de Economía
  y Competitividad} and the European Regional Development Fund
(ERDF/FEDER), project MTM2014-52056-P. A.~Molino was partially
supported by MINECO - FEDER Grant MTM2015-68210-P (Spain), Junta de
Andaluc\'ia FQM-116 (Spain) and MINECO Grant BES-2013-066595 (Spain).

\bibliographystyle{plainnat-linked-initials}
\bibliography{bibliography}

\begin{thebibliography}{33}
\providecommand{\natexlab}[1]{#1}
\providecommand{\url}[1]{\texttt{#1}}
\expandafter\ifx\csname urlstyle\endcsname\relax
  \providecommand{\doi}[1]{doi: #1}\else
  \providecommand{\doi}{doi: \begingroup \urlstyle{rm}\Url}\fi

\bibitem[Andreu-Vaillo et~al.(2010)Andreu-Vaillo, Maz\'{o}n, Rossi, and
  Toledo-Melero]{AMRT2010}
Andreu-Vaillo, F., Maz\'{o}n, J.~M., Rossi, J.~D., and Toledo-Melero, J.~J.
\newblock \href{http://www.worldcat.org/isbn/9780821852309}{\emph{Nonlocal
  diffusion problems}}.
\newblock American Mathematical Society ; Real Sociedad Matem\'{a}tica
  Espa\~{n}ola, 2010.
\newblock ISBN 9780821852309.

\bibitem[Arnold et~al.(2004)Arnold, Carrillo, Desvillettes, Dolbeault,
  J\"{u}ngel, Lederman, Markowich, Toscani, and Villani]{citeulike:2859252}
Arnold, A., Carrillo, J.~A., Desvillettes, L., Dolbeault, J., J\"{u}ngel, A.,
  Lederman, C., Markowich, P.~A., Toscani, G., and Villani, C.
\newblock \href{http://dx.doi.org/10.1007/s00605-004-0239-2}{Entropies and
  equilibria of {Many-Particle} systems: An essay on recent research}.
\newblock \emph{Monatshefte f\"{u}r Mathematik}, 142\penalty0 (1):\penalty0
  35--43, June 2004.

\bibitem[Bakry and \'{E}mery(1985)]{citeulike:8176023}
Bakry, D. and \'{E}mery, M.
\newblock \href{http://dx.doi.org/10.1007/bfb0075847}{Diffusions
  hypercontractives}.
\newblock In Az\'{e}ma, J. and Yor, M., editors, \emph{S\'{e}minaire de
  Probabilit\'{e}s XIX 1983/84}, volume 1123 of \emph{Lecture Notes in
  Mathematics}, chapter~13, pages 177--206. Springer Berlin / Heidelberg, 1985.
\newblock ISBN 978-3-540-15230-9.

\bibitem[Bonforte et~al.(2010)Bonforte, Dolbeault, Grillo, and
  V\'{a}zquez]{citeulike:8175960}
Bonforte, M., Dolbeault, J., Grillo, G., and V\'{a}zquez, J.~L.
\newblock \href{http://dx.doi.org/10.1073/pnas.1003972107}{Sharp rates of decay
  of solutions to the nonlinear fast diffusion equation via functional
  inequalities}.
\newblock \emph{Proceedings of the National Academy of Sciences}, 107\penalty0
  (38):\penalty0 16459--16464, September 2010.

\bibitem[Br\"{a}ndle and de~Pablo(2015)]{BP13}
Br\"{a}ndle, C. and de~Pablo, A.
\newblock \href{http://arxiv.org/abs/1312.4661}{Nonlocal heat equations: decay
  estimates and {Nash} inequalities}, November 2015,
  \href{http://arxiv.org/abs/1312.4661}{{\ttfamily arXiv:1312.4661}}.

\bibitem[Carlen et~al.(1987)Carlen, Kusuoka, and Stroock]{Carlen1987}
Carlen, E.~A., Kusuoka, S., and Stroock, D.~W.
\newblock \href{http://eudml.org/doc/77309}{Upper bounds for symmetric {Markov}
  transition functions}.
\newblock \emph{Annales de l'Institute Henri Poincar\'e. Probabilit\'es et
  statistiques}, 23\penalty0 (S2):\penalty0 245--287, 1987.

\bibitem[Carrillo et~al.(2001)Carrillo, J\"{u}ngel, Markowich, Toscani, and
  Unterreiter]{CJMTU}
Carrillo, J.~A., J\"{u}ngel, A., Markowich, P.~A., Toscani, G., and
  Unterreiter, A.
\newblock \href{http://dx.doi.org/10.1007/s006050170032}{Entropy dissipation
  methods for degenerate parabolic problems and generalized {Sobolev}
  inequalities}.
\newblock \emph{Monatshefte f\"{u}r Mathematik}, 133\penalty0 (1):\penalty0
  1--82, May 2001.
\newblock ISSN 0026-9255.

\bibitem[Chafa\"{i}(2004)]{Chafai04}
Chafa\"{i}, D.
\newblock \href{http://projecteuclid.org/euclid.kjm/1250283556}{Entropies,
  convexity, and functional inequalities}.
\newblock \emph{Journal of Mathematics of Kyoto University}, 44\penalty0
  (2):\penalty0 325--363, October 2004,
  \href{http://arxiv.org/abs/math/0211103}{{\ttfamily arXiv:math/0211103}}.

\bibitem[Chasseigne et~al.(2006)Chasseigne, Chaves, and Rossi]{CCR06}
Chasseigne, E., Chaves, M., and Rossi, J.~D.
\newblock \href{http://dx.doi.org/10.1016/j.matpur.2006.04.005}{Asymptotic
  behavior for nonlocal diffusion equations}.
\newblock \emph{Journal de Math\'{e}matiques Pures et Appliqu\'{e}es},
  86\penalty0 (3):\penalty0 271--291, September 2006.
\newblock ISSN 00217824.

\bibitem[Cort\'{a}zar et~al.(2007)Cort\'{a}zar, Coville, Elgueta, and
  Mart\'{i}­nez]{CCEM}
Cort\'{a}zar, C., Coville, J., Elgueta, M., and Mart\'{i}­nez, S.
\newblock
  \href{http://dx.doi.org/http://dx.doi.org/10.1016/j.jde.2007.06.002}{A
  nonlocal inhomogeneous dispersal process}.
\newblock \emph{Journal of Differential Equations}, 241\penalty0 (2):\penalty0
  332 -- 358, 2007.
\newblock ISSN 0022-0396.

\bibitem[Cort\'{a}zar et~al.(2008)Cort\'{a}zar, Elgueta, Rossi, and
  Wolanski]{CERW2008}
Cort\'{a}zar, C., Elgueta, M., Rossi, J.~D., and Wolanski, N.
\newblock \href{http://dx.doi.org/10.1007/s00205-007-0062-8}{How to approximate
  the heat equation with {Neumann} boundary conditions by nonlocal diffusion
  problems}.
\newblock \emph{Archive for Rational Mechanics and Analysis}, 187\penalty0
  (1):\penalty0 137--156, 2008.
\newblock ISSN 1432-0673.

\bibitem[Cort\'{a}zar et~al.(2009)Cort\'{a}zar, Elgueta, and Rossi]{CER2009}
Cort\'{a}zar, C., Elgueta, M., and Rossi, J.~D.
\newblock \href{http://dx.doi.org/10.1007/s11856-009-0019-8}{Nonlocal diffusion
  problems that approximate the heat equation with {Dirichlet} boundary
  conditions}.
\newblock \emph{Israel Journal of Mathematics}, 170\penalty0 (1):\penalty0
  53--60, 2009.
\newblock ISSN 1565-8511.

\bibitem[Cort{\'a}zar et~al.(2011)Cort{\'a}zar, Elgueta, Garc{\'i}a-Meli{\'a}n,
  and Mart{\'i}nez]{CEG-MM2011}
Cort{\'a}zar, C., Elgueta, M., Garc{\'i}a-Meli{\'a}n, J., and Mart{\'i}nez, S.
\newblock Stationary sign changing solutions for an inhomogeneous nonlocal
  problem.
\newblock \emph{Indiana Univ. Math. J.}, 60:\penalty0 209--232, 2011.
\newblock ISSN 0022-2518.

\bibitem[Cort\'{a}zar et~al.(2015)Cort\'{a}zar, Elgueta, Garc\'{i}a-Meli\'{a}n,
  and Mart\'{i}­nez]{CEG-MM}
Cort\'{a}zar, C., Elgueta, M., Garc\'{i}a-Meli\'{a}n, J., and Mart\'{i}­nez, S.
\newblock \href{http://dx.doi.org/10.3934/dcds.2015.35.1409}{Finite mass
  solutions for a nonlocal inhomogeneous dispersal equation}.
\newblock \emph{Discrete and Continuous Dynamical Systems}, 35\penalty0
  (4):\penalty0 1409--1419, 2015.
\newblock ISSN 1078-0947.

\bibitem[Cort{\'a}zar et~al.(2016)Cort{\'a}zar, Elgueta, Garc{\'i}a-Meli{\'a}n,
  and Mart{\'i}nez]{CEG-MM2016}
Cort{\'a}zar, C., Elgueta, M., Garc{\'i}a-Meli{\'a}n, J., and Mart{\'i}nez, S.
\newblock \href{http://dx.doi.org/10.1007/s00028-015-0299-x}{An inhomogeneous
  nonlocal diffusion problem with unbounded steps}.
\newblock \emph{Journal of Evolution Equations}, 16\penalty0 (1):\penalty0
  209--232, 2016.
\newblock ISSN 1424-3202.

\bibitem[Desvillettes and Villani(2004)]{DV2004}
Desvillettes, L. and Villani, C.
\newblock \href{http://dx.doi.org/10.1007/s00222-004-0389-9}{On the trend to
  global equilibrium for spatially inhomogeneous kinetic systems: the
  {Boltzmann} equation}.
\newblock \emph{Inventiones mathematicae}, 159\penalty0 (2):\penalty0 245--316,
  September 2004.
\newblock ISSN 0020-9910.

\bibitem[Ethier and Kurtz(1986)]{ethier-kurtz86}
Ethier, S.~N. and Kurtz, T.~G.
\newblock
  \href{http://www.amazon.com/exec/obidos/redirect?tag=citeulike07-20\&path=ASIN/0471081868}{\emph{Markov
  Processes: Characterization and Convergence}}.
\newblock Wiley Series in Probability and Statistics. Wiley, March 1986.
\newblock ISBN 0471081868.

\bibitem[Giga et~al.(2010)Giga, Giga, and Saal]{Giga}
Giga, M.-H., Giga, Y., and Saal, J.
\newblock \href{http://dx.doi.org/10.1007/978-0-8176-4651-6}{\emph{Nonlinear
  partial differential equations}}, volume~79 of \emph{Progress in Nonlinear
  Differential Equations and their Applications}.
\newblock Birkh\"auser Boston, Inc., Boston, MA, 2010.
\newblock ISBN 978-0-8176-4173-3.
\newblock Asymptotic behavior of solutions and self-similar solutions.

\bibitem[Gross(1975)]{Gross}
Gross, L.
\newblock \href{http://www.jstor.org/stable/2373688}{Logarithmic sobolev
  inequalities}.
\newblock \emph{American Journal of Mathematics}, 97\penalty0 (4):\penalty0
  1061--1083, 1975.
\newblock ISSN 00029327, 10806377.

\bibitem[Ignat and Rossi(2007)]{IR2007}
Ignat, L.~I. and Rossi, J.~D.
\newblock \href{http://dx.doi.org/10.1016/j.jfa.2007.07.013}{A nonlocal
  convection-diffusion equation}.
\newblock \emph{Journal of Functional Analysis}, 251\penalty0 (2):\penalty0
  399--437, October 2007.
\newblock ISSN 00221236.

\bibitem[Ignat and Rossi(2008)]{citeulike:3617261}
Ignat, L.~I. and Rossi, J.~D.
\newblock \href{http://dx.doi.org/10.1007/s00028-008-0372-9}{Refined asymptotic
  expansions for nonlocal diffusion equations}.
\newblock \emph{Journal of Evolution Equations}, 8\penalty0 (4):\penalty0
  617--629, 2008.

\bibitem[Ignat and Rossi(2009)]{citeulike:5333882}
Ignat, L.~I. and Rossi, J.~D.
\newblock \href{http://dx.doi.org/10.1016/j.matpur.2009.04.009}{Decay estimates
  for nonlocal problems via energy methods}.
\newblock \emph{Journal de Math\'{e}matiques Pures et Appliqu\'{e}es},
  92\penalty0 (2):\penalty0 163--187, August 2009.
\newblock ISSN 00217824.

\bibitem[Michel et~al.(2004)Michel, Mischler, and Perthame]{MMP2004-general}
Michel, P., Mischler, S., and Perthame, B.
\newblock \href{http://dx.doi.org/10.1016/j.crma.2004.03.006}{General entropy
  equations for structured population models and scattering}.
\newblock \emph{Comptes Rendus Mathematique}, 338\penalty0 (9):\penalty0
  697--702, May 2004.
\newblock ISSN 1631073X.

\bibitem[Michel et~al.(2005)Michel, Mischler, and Perthame]{MMP}
Michel, P., Mischler, S., and Perthame, B.
\newblock
  \href{http://dx.doi.org/http://dx.doi.org/10.1016/j.matpur.2005.04.001}{General
  relative entropy inequality: an illustration on growth models}.
\newblock \emph{Journal de Math\'{e}matiques Pures et Appliqu\'{e}es},
  84\penalty0 (9):\penalty0 1235 -- 1260, 2005.
\newblock ISSN 0021-7824.

\bibitem[Mischler and Tristani(2016)]{MT2016}
Mischler, S. and Tristani, I.
\newblock \href{http://arxiv.org/abs/1507.04861}{Uniform semigroup spectral
  analysis of the discrete, fractional \& classical {Fokker-Planck} equations},
  March 2016, \href{http://arxiv.org/abs/1507.04861}{{\ttfamily
  arXiv:1507.04861}}.

\bibitem[Molino and Rossi(2016)]{MR2016}
Molino, A. and Rossi, J.~D.
\newblock \href{http://dx.doi.org/10.1007/s00033-016-0649-8}{Nonlocal diffusion
  problems that approximate a parabolic equation with spatial dependence}.
\newblock \emph{Zeitschrift f{\"u}r angewandte Mathematik und Physik},
  67\penalty0 (3):\penalty0 41, 2016.
\newblock ISSN 1420-9039.

\bibitem[Nash(1958)]{nash1958continuity}
Nash, J.
\newblock \href{http://dx.doi.org/10.2307/2372841}{Continuity of solutions of
  parabolic and elliptic equations}.
\newblock \emph{American Journal of Mathematics}, 80\penalty0 (4):\penalty0
  931--954, October 1958.
\newblock ISSN 0002-9327.

\bibitem[Otto and Villani(2000)]{MR1760620}
Otto, F. and Villani, C.
\newblock \href{http://dx.doi.org/10.1006/jfan.1999.3557}{Generalization of an
  inequality by {T}alagrand and links with the logarithmic {S}obolev
  inequality}.
\newblock \emph{J. Funct. Anal.}, 173\penalty0 (2):\penalty0 361--400, 2000.
\newblock ISSN 0022-1236.

\bibitem[Rey and Toscani(2013)]{Toscani}
Rey, T. and Toscani, G.
\newblock \href{http://dx.doi.org/10.1137/120876290}{Large-time behavior of the
  solutions to {Rosenau}-type approximations to the heat equation}.
\newblock \emph{SIAM Journal on Applied Mathematics}, 73\penalty0 (4):\penalty0
  1416--1438, 2013.

\bibitem[Schonbek(1980)]{Schonbek}
Schonbek, M.~E.
\newblock \href{http://dx.doi.org/10.1080/0360530800882145}{Decay of solution
  to parabolic conservation laws}.
\newblock \emph{Communications in Partial Differential Equations}, 5\penalty0
  (4):\penalty0 449--473, 1980.

\bibitem[Sun et~al.(2011)Sun, Li, and Yang]{SunLiYang2011}
Sun, J.-W., Li, W.-T., and Yang, F.-Y.
\newblock Approximate the {Fokker}-{Planck} equation by a class of nonlocal
  dispersal problems.
\newblock \emph{Nonlinear Analysis: Theory, Methods \& Applications},
  74\penalty0 (11):\penalty0 3501--3509, 2011.

\bibitem[Toscani(2017)]{T2017}
Toscani, G.
\newblock \href{http://arxiv.org/abs/1703.10909}{A {Rosenau}-type approach to
  the approximation of the linear {Fokker-Planck} equation}, March 2017,
  \href{http://arxiv.org/abs/1703.10909}{{\ttfamily arXiv:1703.10909}}.

\bibitem[Villani(2002)]{Villani02}
Villani, C.
\newblock \href{http://dx.doi.org/10.1016/S1874-5792(02)80004-0}{A review of
  mathematical topics in collisional kinetic theory}.
\newblock In Friedlander, S. and Serre, D., editors, \emph{Handbook of
  Mathematical Fluid Dynamics, Vol. 1}, pages 71--305. Elsevier, Amsterdam,
  Netherlands; Boston, U.S.A., 2002.

\end{thebibliography}

\end{document}